\documentclass[11pt]{amsart}

\usepackage[latin1]{inputenc}
\usepackage[english]{babel}
\usepackage{amsmath}
\usepackage{amsfonts}
\usepackage{amssymb}
\usepackage{amsthm}
\usepackage{epsfig}
\usepackage{graphicx,colortbl}
\usepackage{url,hyperref}
\usepackage{mathtools, color, changes}

\newtheorem{theorem}{Theorem}[section]

\newtheorem{corollary}{Corollary}[section]

\newtheorem{remark}{Remark}[section]

\newcounter{theor}

\newtheorem{thm}[theor]{Theorem}

\def\conv{\mathop\mathrm{conv}\nolimits}

\def\R{\mathbb{R}}
\def\N{\mathbb{N}}
\def\Z{\mathbb{Z}}

\def\vol{\mathrm{vol}}

\newcommand{\dlat}{\mathrm{d}}

\def\esc#1{\left\langle #1\right\rangle}

\def\ceil#1{\lceil#1\rceil}
\def\floor#1{\lfloor#1\rfloor}
\def\2Z{2^{-m}\Z^n}
\def\Gsub#1{\mathrm{G}_{#1}}
\def\G{\mathrm{G}_n}
\def\symbol{\diamond}
\def\Media#1#2#3#4{\mathcal{M}_{#1}^{#4}\left(#2,#3\right)}
\def\Suma#1#2#3#4{\mathcal{S}_{#1}^{#4}\left(#2,#3\right)}
\def\Rm{\mathrm{R}^m}
\def\Om{\mathrm{O}^m}

\numberwithin{equation}{section}

\begin{document}

\title[On discrete $L_p$ Brunn-Minkowski type inequalities]{On discrete $L_p$ Brunn-Minkowski type inequalities}

\author{Mar\'\i a A. Hern\'andez Cifre}
\author{Eduardo Lucas}
\author{Jes\'us Yepes Nicol\'as}
\address{Departamento de Matem\'aticas, Universidad de Murcia, Campus de
Espinar\-do, 30100-Murcia, Spain}
\email{mhcifre@um.es}
\email{eduardo.lucas@um.es}
\email{jesus.yepes@um.es}

\thanks{The work is partially supported by MICINN/FEDER project PGC2018-097046-B-I00
and by ``Programa de Ayudas a Grupos de Excelencia de la Regi\'on de Murcia'', Fundaci\'on S\'eneca,
19901/GERM/15.}



\begin{abstract}
$L_p$ Brunn-Minkowski type inequa\-li\-ties for the lattice point enumerator $\G(\cdot)$ are shown, both in a geometrical and in a functional setting.
In particular, we prove that \[\G\bigl((1-\lambda)\cdot K +_p \lambda\cdot L + (-1,1)^n\bigr)^{p/n}\geq (1-\lambda)\G(K)^{p/n}+\lambda\G(L)^{p/n}\] for any $K, L\subset\R^n$ bounded sets with integer points and all $\lambda\in(0,1)$.
We also show that these new discrete analogues (for $\G(\cdot)$) imply the corresponding results concerning the Lebesgue measure.
\end{abstract}

\maketitle

\section{Introduction}

The classical \emph{Brunn-Minkowski inequality} for
non-empty compact subsets $K, L$ of the $n$-dimensional Euclidean space $\R^n$ asserts
that, for any $\lambda\in(0,1)$,
\begin{equation}\label{e:BM}
\vol\bigl((1-\lambda) K+\lambda L\bigr)^{1/n}\geq
(1-\lambda)\vol(K)^{1/n}+\lambda\vol(L)^{1/n}.
\end{equation}
Here $\vol(\cdot)$ denotes the $n$-dimensional Lebesgue measure
(when integrating, $\dlat x$ will stand for $\dlat \vol(x)$) and $+$ is used for the \emph{Minkowski addition}, i.e., $A+B=\{a+b:\, a\in A, \, b\in B\}$ for any non-empty sets $A, B\subset\R^n$. Moreover, $\lambda A$ represents the set
$\{\lambda a:\, a\in A\}$, for $\lambda\geq0$.

The Brunn-Minkowski inequality has become not only a cornerstone of the
Brunn-Minkowski theory (for which we refer the reader to the updated
monograph \cite{Sch2}) but also a powerful tool in other related fields of
mathematics. Among other analogues of it we emphasize its analytic version,
the so-called \emph{Borell-Brascamp-Lieb inequality},
which implies a whole uniparametric family of Brunn-Minkowski type
inequalities.  For extensive survey articles on this and other related inequalities
we refer the reader to \cite{Brt,G}.

When dealing with convex bodies (compact convex sets) $K,L\subset\R^n$ containing the origin,
the following generalization of the classical Min\-kows\-ki addition, usually
referred to as the $p$\emph{-sum} $K+_pL$ of $K$ and $L$, was introduced by Firey \cite{F62}:
for $1\leq p\leq\infty$ fixed, there exists a (unique) convex body $K+_p L$ whose support
function is given by
\begin{equation}\label{e:def_p-sum}
h(K+_p L,\cdot)= \bigl(h(K,\cdot)^p + h(L,\cdot)^p\bigr)^{1/p}.
\end{equation}
When $p=\infty$ this must be interpreted as its limit case, i.e.,
$h(K+_\infty L,\cdot)=\max\bigl\{h(K,\cdot), h(L,\cdot)\bigr\}$, as is customary.
We recall that the support function of a convex body $K\subset\R^n$ is defined by $h(K,u)=\max\bigl\{\esc{x,u}:x\in K\bigr\}$, for all $u\in\R^n$ (see e.g. \cite[Section~1.7]{Sch2}).
One may also define a $p$-scalar multiplication by
$\lambda \cdot_p K:=\lambda^{1/p}K$, for any $\lambda\geq0$.
We observe that although this notion surely depends on $p$, we will use the notation $\cdot$ (instead of $\cdot_p$) throughout the manuscript when this scalar multiplication is used together with the $p$-sum $+_p$.
Moreover, given $\lambda,\mu\geq0$, we write $\lambda \cdot K+_p\mu \cdot L$ for $(\lambda \cdot K)+_p(\mu \cdot L)$.
Clearly, when $p=1$ the latter set recovers the classical linear combination $\lambda K+\mu L$ (cf. \eqref{e:def_p-sum}), whereas the case $p=\infty$ yields
\[
\lambda\cdot K+_{\infty}\mu\cdot L=\conv(K\cup L).
\]

The main disadvantage of the previous definition of $p$-sum is that it is defined via $p$-means of the support functions of the convex bodies (containing the origin) there involved, which implies the necessity of assuming convexity, unlike what happens for the usual Minkowski sum. Lutwak, Yang and Zhang \cite{LYZ} extended the $p$-sum, for $1\leq p<\infty$, to the case of arbitrary subsets of the Euclidean space, by showing that there is a pointwise definition of it, similar to that of the Minkowski addition: for any $K,L\subset\R^n$,
\begin{equation}\label{e:p-sum_extended}
K+_pL:=\left\{(1-\mu)^{1/q}x+\mu^{1/q}y:x\in K,y\in L, \, 0\leq\mu\leq1\right\},
\end{equation}
where $q$ is the H\"older conjugate of $p$ (i.e., such that
$1/p+1/q=1$). From now on, given $p\geq1$, the notation $q$ will have this meaning, unless stated otherwise.

In \cite{LYZ} it is shown that the definition in \eqref{e:p-sum_extended} coincides with the one given by \eqref{e:def_p-sum} when $K$ and $L$ are $n$-dimensional convex bodies containing the origin. Moreover, in the case when $p=1$ (and hence $q=\infty$), the coefficients $(1-\mu)^{1/q},\mu^{1/q}$ must be understood as $1$ for all $0\leq\mu\leq 1$, and thus $K+_1L$ equals $K+L$. Furthermore, as shown in \cite{LYZ}, one has
\begin{equation*}\label{e:1-mean<p-mean}
(1-\lambda)\cdot K+_p\lambda\cdot L \supset(1-\lambda)K+\lambda L
\end{equation*}
for all $0\leq\lambda\leq 1$.
Finally, we would like to mention that, although \eqref{e:p-sum_extended} makes also sense for $p=\infty$ (and so $q=1$), we will omit this case throughout the manuscript (following \cite{LYZ} too), since for such a value of $p$ all the results trivially hold. So, along the rest of the paper, when writing $p\geq1$ we will refer to a real number $p\geq1$.

The $L_p$ version of the Brunn-Minkowski inequality \eqref{e:BM} was originally proven by Firey \cite{F62},
in the setting of convex bodies containing the origin, and by Lutwak, Yang and Zhang (see~\cite[Theorem~4]{LYZ}) for arbitrary non-empty compact sets:
\begin{thm}\label{t:lpbm}
Let $\lambda\in(0,1)$ and $p\geq 1$, and let $K,L\subset\R^n$ be non-empty compact sets. Then
\begin{equation}\label{ineq:lpbm}
\vol\bigl((1-\lambda)\cdot K+_p\lambda\cdot L\bigr)^{p/n}\geq(1-\lambda)\vol(K)^{p/n}+\lambda\vol(L)^{p/n}.
\end{equation}
\end{thm}

Around three decades after the introduction given by Firey for the $p$-sum of convex bodies (containing the origin),
Lutwak \cite{L93,L96} initiated a deep and systematic study of $p$-additions and their consequences.
This new and remarkable extension of the classical Brunn-Minkowski theory, usually referred to in the literature as the $L_p$ \emph{Brunn-Minkowski theory}, is not only a very active area of research nowadays, but it has further supposed to be the starting point for new developments and generalizations. An example of the latter can be seen in \cite{GHW,GHW2,Me} and the references therein, where the authors perform a thorough investigation into the fundamental characte\-ristics of operations between sets and provide with an elegant construction that allows one to define a general pointwise operation between sets. For more information on the $L_p$ \emph{Brunn-Minkowski theory} and its consequences we refer the reader to \cite[Section~9.1]{Sch2}.

\smallskip

In the discrete setting of $\Z^n$ endowed with the cardinality $|\cdot|$, Gardner and Gronchi \cite{GG}
obtained an engaging and powerful analogue of the following form of the Brunn-Minkowski inequality: $\vol(K+L)\geq\vol(B_K+B_L)$, where $B_K$ and $B_L$ are centered Euclidean balls of the same volume as
the convex bodies $K$ and $L$, respectively. Moreover, from the above-mentioned discrete version, they derive
some inequalities that improve previous results obtained by Ruzsa in
\cite{Ru1,Ru2}.

More recently, different discrete analogues of the Brunn-Minkowski
ine\-quality have been obtained, including the case of its classical form
(cf. \eqref{e:BM}) for the cardinality \cite{GT,HCIYN,IYNZ}, functional
extensions of it \cite{HKS,IYN,IYNZ,KL,Sl} and versions for the \emph{lattice
point enumerator} $\G(\cdot)$ \cite{HKS,ILYN,IYNZ}, which is defined by
$\G(M)=|M\cap\Z^n|$, $M\subset\R^n$. In this respect, in \cite{IYNZ} it is shown the necessity of extending
$(1-\lambda)K+\lambda L$ to $(1-\lambda)K+\lambda L+(-1,1)^n$ in order to get a discrete analogue of \eqref{e:BM}
for all $\lambda\in(0,1)$, as follows:
\begin{thm}\label{t: BM_lattice_point_no_G(K)G(L)>0}
Let $\lambda\in(0,1)$ and let $K,L\subset\R^n$ be non-empty bounded sets.
Then
\begin{equation}\label{e: BM_lattice_point_no_G(K)G(L)>0}
\G\bigl((1-\lambda)K+\lambda L+(-1,1)^n\bigr)^{1/n}\geq(1-\lambda)\G(K)^{1/n}+\lambda\G(L)^{1/n}.
\end{equation}
The inequality is sharp.
\end{thm}

Here we are mainly interested in finding a discrete counterpart to \eqref{ineq:lpbm}, or equivalently,
in getting an $L_p$ version of \eqref{e: BM_lattice_point_no_G(K)G(L)>0}. In this regard, we show the following:
\begin{theorem}\label{t:p-B-M_discrete}Let $\lambda\in(0,1)$ and $p\geq 1$, and let $K,L\subset\R^n$ be bounded sets with $\G(K)\G(L)>0$. Then
\begin{equation}\label{ineq:dlpbm}
\G\bigl((1-\lambda)\cdot K+_p\lambda\cdot
L+(-1,1)^n\bigr)^{p/n}\geq(1-\lambda)\G(K)^{p/n}+\lambda\G(L)^{p/n}.
\end{equation}
The inequality is sharp.
\end{theorem}

For any fixed $p\geq1$, the Minkowski addition of the cube $(-1,1)^n$ on the left-hand side of the latter inequality cannot be, in general, neither reduced (by means of a smaller cube) nor substituted by its $p$-sum (see Remark~\ref{r:(-1,1)^n_cannot_be_reduced}). And again, as in the classical framework, the case of
$p=1$ of this result reco\-vers \eqref{e: BM_lattice_point_no_G(K)G(L)>0}. Furthermore, we
show that the $L_p$ Brunn-Minkowski inequality \eqref{ineq:lpbm}, in the setting of $n$-dimensional convex
bodies, can be derived as a consequence of this new discrete inequality for the lattice point enumerator $\G(\cdot)$:
\begin{theorem}\label{t:L_pBM_disc_to cont}
The discrete $L_p$ Brunn-Minkowski type inequality \eqref{ineq:dlpbm} implies the $L_p$ Brunn-Minkowski inequality \eqref{ineq:lpbm} for $n$-dimensional convex bo\-dies $K$ and $L$.
\end{theorem}
In fact, we will prove these results on the lattice point enumerator $\G(\cdot)$ by showing (the more general version of) their functional counterpart (see Theorems~\ref{t:Lp-BBL_discreta} and~\ref{t:bblp_disc_to_cont}).

\smallskip

The paper is organized as follows: in Section \ref{s:main} we recall some preliminaries and we state our main results (in the functional setting), whereas their proofs will be established in Section \ref{s:proofs}.

\section{Functional results: Background and main results}\label{s:main}

As mentioned before, we will obtain Theorem \ref{t:p-B-M_discrete} as a direct consequence of its functional analogue. To introduce it, we recall the analytical counterpart (for functions) of the Brunn-Minkowski inequality, the so-called \emph{Borell-Brascamp-Lieb inequality}, originally proven in
\cite{Borell} and \cite{BL}. For its statement, we first need to give the definition of the $\alpha$-sum
$\Suma{\alpha}{\cdot}{\cdot}{t,s}$ of two non-negative numbers, with positive coefficients $t$ and $s$,
where $\alpha$ is a parameter varying in $\R\setminus\{0\}\cup\{\pm\infty\}$, as well as the notion of $\alpha$-mean
$\Media{\alpha}{\cdot}{\cdot}{\lambda}$, with $\lambda\in(0,1)$, for $\alpha\in\R\cup\{\pm\infty\}$
(for a general reference for $\alpha$-sums and means of non-negative numbers, we refer the reader to the classic text of Hardy, Littlewood and P\'olya \cite{HaLiPo} and to the handbook \cite{Bu}). We consider first
the case $\alpha\in\R$, with $\alpha\neq0$: given $a,b>0$, let
\[
\Suma{\alpha}{a}{b}{t,s}=(ta^\alpha+s b^\alpha\bigr)^{1/\alpha}.
\]
For $\alpha=\pm \infty$ we set $\Suma{\infty}{a}{b}{t,s}=\max\{a,b\}$ and
$\Suma{-\infty}{a}{b}{t,s}=\min\{a,b\}$. Furthermore, if $ab=0$, we define $\Suma{\alpha}{a}{b}{t,s}=0$ for all $\alpha\in\R\setminus\{0\}\cup\{\pm\infty\}$, and moreover,
when $t=s=1$ we just write
\[\Suma{\alpha}{a}{b}{}=\Suma{\alpha}{a}{b}{1,1}.\]
Finally, for any $\alpha\neq0$ we set \[\Media{\alpha}{a}{b}{\lambda}=\Suma{\alpha}{a}{b}{1-\lambda,\lambda}\] whereas for $\alpha=0$ we write $\Media{0}{a}{b}{\lambda}=a^{1-\lambda}b^{\lambda}$.


\smallskip

The reason to modify in this way (when $ab=0$) the definition of $\alpha$-sums given in \cite{HaLiPo} is due to the classical statement of the Borell-Brascamp-Lieb inequality, which is
collected below. In fact, without such a modification (although redundant for any $\alpha\leq0$),
if we do not assume $f(x)g(y)>0$ in \eqref{e:BBL_means_hyp}, the thesis of this result would not have mathematical interest when $\alpha>0$.

\begin{thm}[The Borell-Brascamp-Lieb inequality]\label{t:BBL_means}
Let $\lambda\in(0,1)$. Let $-1/n\leq \alpha\leq\infty$ and let $f,g,h:\R^n\longrightarrow\R_{\geq0}$ be integrable
functions such that
\begin{equation}\label{e:BBL_means_hyp}
h\bigl((1-\lambda)x +\lambda y\bigr)\\
 \geq\bigl[(1-\lambda)f(x)^{\alpha}+\lambda g(y)^{\alpha}\bigr]^{1/\alpha}
\end{equation}
for all $x,y\in\R^n$ with $f(x)g(y)>0$. Then
\begin{equation*}\label{e:BBL_means}
\int_{\R^n}h(x)\,\dlat x\geq
\Media{\frac{\alpha}{n\alpha+1}}{\int_{\R^n}f(x)\,\dlat x}{\int_{\R^n}g(x)\,\dlat x}{\lambda}.
\end{equation*}
\end{thm}

Taking into account the definition of $p$-sum given by \eqref{e:p-sum_extended}, it is natural to wonder about the possibility of extending the above result to the $L_p$ setting by suitably modifying the condition on the functions there involved (cf. \eqref{e:BBL_means_hyp}). Such an expected $L_p$ version of the Borell-Brascamp-Lieb inequality
has been very recently obtained in \cite{RoXi}  (shown independently, for the case of $\alpha>0$, in \cite{Wu}):
\begin{thm}\label{t:L_p-BBL}
Let $\lambda\in(0,1)$ and $p\geq1$. Let $-1/n\leq\alpha\leq\infty$ and let $f,g,h:\R^n\longrightarrow\R_{\geq0}$ be integrable functions such that
\begin{equation}\label{e:L_p-BBL_hyp}
\begin{split}
h\Bigl((1-\lambda)^{1/p}(1-\mu)^{1/q}x & +\lambda^{1/p}\mu^{1/q}y\Bigr)\\
 & \geq\Bigl[(1-\lambda)^{1/p}(1-\mu)^{1/q}f(x)^{\alpha}+\lambda^{1/p}\mu^{1/q}g(y)^{\alpha}\Bigr]^{1/\alpha}
\end{split}
\end{equation}
for all $x,y\in\R^n$ with $f(x)g(y)>0$ and all $\mu\in[0,1]$. Then
\begin{equation*}\label{e:L_p-BBL}
\int_{\R^n}h(x)\,\dlat x\geq
\Media{\frac{p\alpha}{n\alpha+1}}{\int_{\R^n}f(x)\,\dlat x}{\int_{\R^n}g(x)\,\dlat x}{\lambda}.
\end{equation*}
\end{thm}

\smallskip

Before going on, we would like to clarify the special case of $\alpha=0$ in condition \eqref{e:L_p-BBL_hyp}
of the previous result:
\begin{remark}\label{r:alpha=0}
On the one hand, it is clear that
\[\lim_{\alpha\to0^+}\bigl(sa^{\alpha}+rb^{\alpha}\bigr)^{1/\alpha}
=\left\{
\begin{array}{ll}
a^{1-r}b^{r} & \text{ if } s+r=1,\\
0 & \text{ if } s+r<1.
\end{array}
\right.
\]
On the other hand, we have $(1-\lambda)^{1/p}(1-\mu)^{1/q}+\lambda^{1/p}\mu^{1/q}=1$ if $\mu=\lambda$
and \[(1-\lambda)^{1/p}(1-\mu)^{1/q}+\lambda^{1/p}\mu^{1/q}<1\] for all $\mu\in[0,1]$ with $\mu\neq\lambda$, by H\"older's inequality (jointly with its equality case, see e.g.~\cite[Theorem~11]{HaLiPo}). Then, by convention,
the case $\alpha=0$ in \eqref{e:L_p-BBL_hyp} will be understood as
\begin{equation*}
h\bigl((1-\lambda)x +\lambda y\bigr)
\geq f(x)^{1-\lambda}g(y)^{\lambda}
\end{equation*}
for all $x,y\in\R^n$. In other words, the case $\alpha=0$ in Theorem~\ref{t:L_p-BBL} is the same to the one in
Theorem~\ref{t:BBL_means}, i.e., the classical \emph{Pr\'ekopa-Leindler inequality}.
\end{remark}

\smallskip

For the statement of the next result, we first need to introduce some additional notation. From now on we will write $\chi_{_M}$ to represent the characteristic function of a given set $M\subset\R^n$,
namely,
\[\chi_{_M}(x)=\left\{
\begin{array}{ll}
1 & \text{ for } x\in M,\\
0 & \text{ for } x\in\R^n\setminus M.
\end{array}
\right.
\]
Moreover, for a function $\phi:\R^n\longrightarrow\R_{\geq0}$ we denote by $\phi^{\symbol}:\R^n\longrightarrow\R_{\geq0}$ the function defined by	
\[\phi^{\symbol}(z) = \sup_{u\in(-1,1)^n}\phi(z+u)  \quad \text{ for all } z\in\R^n.\]
Such an extension of $\phi$ is just the \emph{Asplund product} $\star$ of the functions $\phi$ and $\chi_{(-1,1)^n}$, which can be seen as the functional analogue of the Minkowski sum of sets in the setting of log-concave functions. Indeed,
\begin{equation*}
\begin{split}
\phi^{\symbol}(z) &= \sup_{u\in(-1,1)^n}\phi(z+u) =\sup_{u \in \R^n} \phi(z+ u) \, \chi_{(-1,1)^n}(-u)\\
&=\sup_{u_1+u_2=z} \phi(u_1) \, \chi_{(-1,1)^n}(u_2) = \bigl(\phi \star \chi_{(-1,1)^n}\bigr)(z).
\end{split}
\end{equation*}
For more information on the Asplund product, also known as the \emph{sup-convolution}, we refer the reader to \cite[Section~9.5]{Sch2} and the references therein.

Taking into account this notation, the following discrete
Borell-Brascamp-Lieb inequality was shown in \cite{IYNZ}:

\begin{thm}\label{t:BBL_discreta}
Let $\lambda\in(0,1)$ and let $K, L \subset \R^n$ be non-empty
bounded sets. Let $-1/n\leq\alpha\leq\infty$ and let $f,g,h:\R^n\longrightarrow \R_{\geq 0}$ be non-negative
functions such that
\[
h\bigl((1-\lambda)x+\lambda y\bigr)\geq
    \bigl[(1-\lambda)f(x)^{\alpha}+\lambda g(y)^{\alpha}\bigr]^{1/\alpha}
\]
for all $x\in K$, $y\in L$ with $f(x)g(y)>0$. Then
\begin{equation*}\label{e:BBL_discreta}
\sum_{z\in (M+(-1,1)^n)\cap\Z^n}\!\!h^\symbol(z)\geq\!
\Media{\frac{\alpha}{n\alpha+1}}{\sum_{x\in K\cap\Z^n}\!\!f(x)}{\sum_{y\in L\cap\Z^n}\!\!g(y)}{\lambda},
\end{equation*}
where $M=(1-\lambda)K+\lambda L$.
\end{thm}

\smallskip

In this paper we show the corresponding $L_p$ version of the latter result. In other words, we prove the following discrete analogue of Theorem~\ref{t:L_p-BBL}, which, in particular, will imply Theorem~\ref{t:p-B-M_discrete}:
\begin{theorem}\label{t:Lp-BBL_discreta}
Let $\lambda\in(0,1)$ and $p\geq 1$, and let $K,L\subset\R^n$ be
non-empty bounded sets. Let $-1/n\leq\alpha\leq\infty$ and let $f,g,h:\R^n\longrightarrow
\R_{\geq 0}$ be non-negative functions such that
\begin{equation}\label{e:Lp-BBL_discreta_hipotesis}
\begin{split}
h\Bigl((1-\lambda)^{1/p}(1-\mu)^{1/q}x & +\lambda^{1/p}\mu^{1/q}y\Bigr)\\
 & \geq\Bigl[(1-\lambda)^{1/p}(1-\mu)^{1/q}f(x)^{\alpha}+\lambda^{1/p}\mu^{1/q}g(y)^{\alpha}\Bigr]^{1/\alpha}
\end{split}
\end{equation}
for all $x\in K$, $y\in L$ with $f(x)g(y)>0$ and all $\mu\in[0,1]$. Then
\begin{equation}\label{e:Lp-BBL_discreta}
\sum_{z\in(M_p+(-1,1)^n)\cap\Z^n}\!\!\!\!\! h^\symbol(z)\geq\!
\Media{\frac{p\alpha}{n\alpha+1}}{\sum_{x\in
K\cap\Z^n}\!\!\!f(x)}{\sum_{y\in L\cap\Z^n}\!\!\!g(y)}{\lambda},
\end{equation}
where $M_p=(1-\lambda)\cdot K+_p\lambda\cdot L$.
\end{theorem}

As in the classical framework, the case $\alpha=0$ in this result is that of Theorem~\ref{t:BBL_discreta} (see Remark~\ref{r:alpha=0}).

We will also show that our discrete counterpart, Theorem~\ref{t:Lp-BBL_discreta}, implies the continuous result collected in Theorem~\ref{t:L_p-BBL}, under mild assumptions for the functions there involved.

\begin{theorem}\label{t:bblp_disc_to_cont}
The discrete $L_p$ Borell-Brascamp-Lieb type inequality (Theorem~\ref{t:Lp-BBL_discreta}) implies the (continuous) $L_p$ Borell-Brascamp-Lieb inequality (Theorem~\ref{t:L_p-BBL}), provided that the functions $f,g$ are Riemann integrable and $h$ is upper semicontinuous.
\end{theorem}

\section{Proofs and further consequences}\label{s:proofs}

To prove the $L_p$ version of the discrete Borell-Brascamp-Lieb
inequa\-lity~\eqref{e:Lp-BBL_discreta} we need to show the following auxiliary result
(here $\ceil{x}$ denotes the ceiling function of $x$, i.e., the least integer greater than or equal to $x$):
\begin{theorem}\label{t:BBL_discreta_ts}
Let $t,s>0$ and let $K,L\subset\R^n$ be
non-empty bounded sets. Let $-1/n\leq\alpha\leq\infty$, $\alpha\neq0$, and let $f,g,h:\R^n\longrightarrow
\R_{\geq 0}$ be non-negative functions such that
\[
h(tx+sy)\geq\bigl[tf(x)^{\alpha}+sg(y)^{\alpha}\bigr]^{1/\alpha}
\]
for all $x\in K$, $y \in L$ with $f(x)g(y)>0$. Then
\begin{equation*}\label{e:BBL_discreta_ts}
\sum_{z\in(M+(-1,\lceil t+s\rceil)^n)\cap\Z^n}h^\symbol(z)\geq
\Suma{\frac{\alpha}{n\alpha+1}}{\sum_{x\in K\cap\Z^n}f(x)}{\sum_{y\in L\cap\Z^n}g(y)}{t,s},
\end{equation*}
where $M=tK+sL$.
\end{theorem}
To show this, we need the statement of the following Brunn-Minkowski type inequality for
the lattice point enumerator, proven in \cite{ILYN}:
\begin{thm}\label{t:BM_tK+sL_(gen_dim)}
Let $t,s\geq 0$ and
let $K,L\subset\R^n$ be bounded sets such that $\G(K)\G(L)>0$. Then
\begin{equation}\label{e:BM_tK+sL_(gen_dim)}
\G\Bigl(tK+sL+\bigl(-1,\lceil t+s\rceil\bigr)^n\Bigr)^{1/n}\geq
t\G(K)^{1/n}+s\G(L)^{1/n}.
\end{equation}
The inequality is sharp.
\end{thm}
The proof of Theorem \ref{t:BBL_discreta_ts} now follows by using the same steps to those of the proof of Theorem~\ref{t:BBL_discreta}, just replacing convex combinations $(1-\lambda)x+\lambda y$, for $\lambda\in(0,1)$, by linear combinations $tx+sy$, with $t,s>0$, and applying \eqref{e:BM_tK+sL_(gen_dim)} instead of
\eqref{e: BM_lattice_point_no_G(K)G(L)>0}.

\medskip

Now we are ready to show our main result. We follow here the underlying idea of the original proof of \eqref{ineq:lpbm} given in \cite{LYZ}.
\begin{proof}[Proof of Theorem~\ref{t:Lp-BBL_discreta}]
Along the proof, we will assume that
\[\left(\sum_{x\in K\cap\Z^n}f(x)\right)\left(\sum_{y\in L\cap\Z^n}g(y)\right)>0,\]
since the result is trivial otherwise.
Now we set, for any given $\mu_0\in[0,1]$ (to be suitably chosen later),
\[
t=t(\mu_0):=(1-\lambda)^{1/p}(1-\mu_0)^{1/q}\quad\text{ and }\quad
s=s(\mu_0):=\lambda^{1/p}\mu_0^{1/q},
\]
for which one has, by H\"older's inequality, that $t+s\leq1$.
Notice that the assumption \eqref{e:Lp-BBL_discreta_hipotesis}
can be then rewritten, in terms
of $t,s$, as
\[
h(tx+sy)\geq\bigl[tf(x)^{\alpha}+sg(y)^{\alpha}\bigr]^{1/\alpha}
\]
for all $x\in K$ and $y\in L$ with $f(x)g(y)>0$, and thus
Theorem~\ref{t:BBL_discreta_ts} yields
\begin{equation}\label{e:applying_BBL_discreta_ts}
\sum_{z\in(tK+sL+(-1,\lceil t+s\rceil)^n)\cap\Z^n}h^\symbol(z)
\geq \Suma{\frac{\alpha}{n\alpha+1}}{\sum_{x\in K\cap\Z^n}f(x)}{\sum_{y\in L\cap\Z^n}g(y)}{t,s}.
\end{equation}
Moreover, from \eqref{e:p-sum_extended} we clearly have
\begin{equation*}\label{e:inclusion_K+pL_mu}
M_p =(1-\lambda)\cdot K+_p\lambda\cdot L
     \supset (1-\lambda)^{1/p}(1-\mu_0)^{1/q} K + \lambda^{1/p}\mu_0^{1/q}L = tK+sL.
\end{equation*}
This, together with \eqref{e:applying_BBL_discreta_ts} and the fact that $(-1,1)^n\supset\bigl(-1,\lceil t+s\rceil\bigr)^n$, allows us to conclude that
\[
\sum_{z\in(M_p+(-1,1)^n)\cap\Z^n}h^\symbol(z)
\geq \Suma{\frac{\alpha}{n\alpha+1}}{\sum_{x\in K\cap\Z^n}f(x)}{\sum_{y\in L\cap\Z^n}g(y)}{t,s}.
\]
Notice also that if $\alpha=-1/n$ then $\alpha/(n\alpha+1)=-\infty$ and hence we are done. Then, in the following we may assume that $\alpha\neq0,-1/n$ (cf. Remark~\ref{r:alpha=0}) and thus, defining $\beta:=\alpha/(n\alpha+1)\in(-\infty,0)\cup(0,1/n]$, we must check whether
\begin{equation}\label{e:suma=media}
\Suma{\beta}{\sum_{x\in K\cap\Z^n}f(x)}{\sum_{y\in L\cap\Z^n}g(y)}{t,s}\geq
\Media{p\beta}{\sum_{x\in K\cap\Z^n}f(x)}{\sum_{y\in L\cap\Z^n}g(y)}{\lambda}
\end{equation}
for a suitable value of $\mu_0\in[0,1]$. To this aim, it is enough to take
\[
\mu_0:=\frac{\lambda\left(\sum_{y\in
L\cap\Z^n}g(y)\right)^{p\beta}}{(1-\lambda)\left(\sum_{\vphantom{y}x\in
K\cap\Z^n}f(x)\right)^{p\beta}+\lambda\left(\sum_{y\in
L\cap\Z^n}g(y)\right)^{p\beta}},
\]
and a straightforward computation shows that \eqref{e:suma=media} indeed holds (in fact, with equality).
This concludes the proof.
\end{proof}

\begin{remark}
Following the same approach as the one in the proof of Theorem~\ref{t:Lp-BBL_discreta}, just replacing the sums of the functions $f$, $g$ and $h$ by their integrals on $\R^n$, one may also derive Theorem~\ref{t:L_p-BBL}. For, one can similarly exploit the suitable version of Theorem \ref{e:BBL_means} for linear combinations $tx+sy$ instead of the one for means $(1-\lambda)x+\lambda y$ (see \cite{Borell}).
\end{remark}

\smallskip

An analogous result for arbitrary lattices can be obtained. We recall that an $n$-dimensional lattice $\Lambda\subset\R^n$ is the set of all integer combinations of $n$ linearly independent vectors $v_1,\dots,v_n$, the set $\mathcal{B}=\{v_1,\dots,v_n\}$ being called a basis of $\Lambda$. Thus, for such an
$n$-dimensional lattice $\Lambda$, let
$\varphi:\R^n\longrightarrow\R^n$ be the linear (bijective) map defined by $\varphi(x)=\sum_{i = 1}^n x_i v_i$ for each $x=(x_1,\dots,x_n)\in\R^n$.
Taking into account the pointwise definition of the $p$-sum given in \eqref{e:p-sum_extended}, we clearly have
\[\varphi\bigl((1-\lambda)\cdot\varphi^{-1}(K)+_p\lambda\cdot\varphi^{-1}(L)\bigr)=(1-\lambda)\cdot K+_p\lambda\cdot L.\]
This allows us to extend the statement of Theorem \ref{t:Lp-BBL_discreta} to the setting of an $n$-dimensional lattice $\Lambda\subset\R^n$, by considering the auxiliary functions
$f_{\mathcal{B}}, g_{\mathcal{B}}, h_{\mathcal{B}}: \R^n \longrightarrow \R_{\geq0}$ given by
\[ f_{\mathcal{B}}(x) = f\bigl(\varphi(x)\bigr), \quad g_{\mathcal{B}}(x) = g\bigl(\varphi(x)\bigr)\,\text{ and } \,\, h_{\mathcal{B}}(x) = h\bigl(\varphi(x)\bigr)\] for any $x\in \R^n$, as follows:
\begin{corollary}\label{c:Lp-BBL_disc_lattice}
Let $\lambda\in(0,1)$ and $p\geq 1$, and let $K,L\subset\R^n$ be
non-empty bounded sets. Let $-1/n\leq\alpha\leq\infty$ and let $f,g,h:\R^n\longrightarrow
\R_{\geq 0}$ be non-negative functions such that
\begin{equation*}
\begin{split}
h\Bigl((1-\lambda)^{1/p}(1-\mu)^{1/q}x & +\lambda^{1/p}\mu^{1/q}y\Bigr)\\
 & \geq\Bigl[(1-\lambda)^{1/p}(1-\mu)^{1/q}f(x)^{\alpha}+\lambda^{1/p}\mu^{1/q}g(y)^{\alpha}\Bigr]^{1/\alpha}
\end{split}
\end{equation*}
for all $x\in K$, $y\in L$ with $f(x)g(y)>0$ and all $\mu\in[0,1]$. Let $\Lambda\subset\R^n$ be an $n$-dimensional lattice with basis $\mathcal{B}=\{v_1,\dots,v_n\}$ and let $\varphi(x)=\sum_{i = 1}^n x_i v_i$ for $x\in\R^n$. Then
\begin{equation*}
\sum_{z\in(M_p+\varphi((-1,1)^n))\cap\Lambda}\!\!\!\!\! h^{\symbol_{_\mathcal{B}}}(z)\geq\!
\Media{\frac{p\alpha}{n\alpha+1}}{\sum_{x\in
K\cap\Lambda}\!\!\!f(x)}{\sum_{y\in L\cap\Lambda}\!\!\!g(y)}{\lambda},
\end{equation*}
where $M_p=(1-\lambda)\cdot K+_p\lambda\cdot L$ and $h^{\symbol_{_\mathcal{B}}}(z) = \sup_{u\in \varphi((-1,1)^n)} h(z+u)$ for all $z\in\R^n$.
\end{corollary}

\smallskip

\subsection{Geometric consequences}
Notice that, as in the classical setting, the geometric inequality \eqref{ineq:dlpbm} can be derived
from the functional one \eqref{e:Lp-BBL_discreta}:
\begin{proof}[Proof of Theorem~\ref{t:p-B-M_discrete}]
By applying \eqref{e:Lp-BBL_discreta} with $\alpha=\infty$ to the characteristic functions $f=\chi_{_K}$, $g=\chi_{_L}$ and $h=\chi_{_{(1-\lambda)\cdot K +_p \lambda\cdot L}}$, for which $h^\symbol=\chi_{_{(1-\lambda)\cdot K+_p\lambda\cdot L+(-1,1)^n}}$, one immediately gets \eqref{ineq:dlpbm}.

\smallskip

Finally, to show that the equality can be attained, it is enough to consider $K = L = [0,m]^n$ with $m \in \N$,
for which $\G\bigl((1-\lambda)\cdot K +_p \lambda L+(-1,1)^n\bigr)=(m+1)^n=\G(K)=\G(L)$.
\end{proof}

For bounded sets $K, L \subset \R^n$ with $\G(K)\G(L)>0$, it was shown in \cite{IYNZ} that
\begin{equation*}\label{e: BM_lattice_point_1/2}
\G\left(\frac{K + L}{2} + [0,1]^n\right)^{1/n}\geq \frac{\G(K)^{1/n}+\G(L)^{1/n}}{2},
\end{equation*}
i.e., that \eqref{e: BM_lattice_point_no_G(K)G(L)>0} for $\lambda=1/2$ also holds by replacing the cube $(-1,1)^n$
by $[0,1]^n$. However, the latter inequality is in general not true for any $\lambda\in(0,1)$.
Thus, and regarding \eqref{ineq:dlpbm}, it is a natural question whether $(-1,1)^n$ might be reduced to a smaller cube.
\begin{remark}\label{r:(-1,1)^n_cannot_be_reduced}
We notice on the one hand that the set $(-1,1)^n$ cannot be reduced to a strictly smaller cube of the form $(-1,a]^n$ (or $[-a,1)^n$) with $a\in(0,1)$, for any fixed value of $p\geq1$.
Indeed, it is enough to consider, as an example, the sets $K=[0,1]$, $L=[0,2]$ in dimension $n=1$
and the combination
\[
M=(1-\lambda)\cdot K+_p\lambda\cdot L+(-1,a]
=\left(-1,\mathcal{M}_p^\lambda(1,2)+a\right]
\]
(observe that $K,L\subset\R^n$ are $n$-dimensional convex bodies containing the origin and hence, as mentioned in the introduction, the $p$-sum defined by \eqref{e:p-sum_extended} agrees with the classical definition given by \eqref{e:def_p-sum}).
Then, since
$\Gsub{1}(M)=\left\lfloor\Media{p}{1}{2}{\lambda}+a\right\rfloor+1$
and $\Media{p}{1}{2}{\lambda}\in[1,2]$, where $\floor{x}$ denotes the floor function of
the real number $x$ (i.e., the greatest integer less than or equal to $x$),
it is enough to find
$\lambda>0$ such that $\Media{p}{1}{2}{\lambda}+a<2$. But this is always
possible because
\[\lim_{\lambda\to0^+} \Media{p}{1}{2}{\lambda}=1,\]
and therefore we have $\Gsub{1}(M)=2$.
However, for the right-hand side of \eqref{ineq:dlpbm} we have $\Gsub{1}(K)=2$ and
$\Gsub{1}(L)=3$, and thus
\[
\Media{p}{\vphantom{\bigl(}\Gsub{1}(K)}{\Gsub{1}(L)}{\lambda}=\Media{p}{2}{3}{\lambda}\in[2,3].
\]
Since $\lambda>0$, we know that
$\Media{p}{\vphantom{\bigl(}\Gsub{1}(K)}{\Gsub{1}(L)}{\lambda}>2$,
which shows that
\[
\Gsub{1}\bigl((1-\lambda)\cdot K+_p\lambda\cdot
L+(-1,a]\bigr)<\Media{p}{\vphantom{\bigl(}\Gsub{1}(K)}{\Gsub{1}(L)}{\lambda}.
\]

On the other hand, taking a look at \eqref{e: BM_lattice_point_no_G(K)G(L)>0},
one could think that its natural $L_p$ version could be given by considering the $p$-sum of the cube $(-1,1)^n$ on
the left-hand side of \eqref{ineq:dlpbm} (instead of its Minkowski addition). In fact, when dealing with $n$-dimensional convex bodies $K,L\subset\R^n$ containing the origin, one has that
\[(1-\lambda)\cdot K+_p\lambda\cdot L+_p[-1,1]^n\subset (1-\lambda)\cdot K+_p\lambda\cdot L+[-1,1]^n\]
for any $p\geq1$ (see \cite{F62}). So, $p$-summing the cube $(-1,1)^n$ on the left-hand side of \eqref{ineq:dlpbm} would be, sometimes, tighter than (Minkowski) adding it.
Ne\-vertheless, this is not possible either. Indeed, by considering again
the sets $K=[0,1]$, $L=[0,2]$ in dimension $n=1$ and $p=2$,
for which we then have by \eqref{e:def_p-sum} that
\[\frac{1}{2}\cdot K +_2 \frac{1}{2}\cdot L=\bigl[0,\sqrt{2.5}\,\bigr],\]
we get, now using \eqref{e:p-sum_extended},
\begin{equation*}
\begin{split}
\Gsub{1}\left(\frac{1}{2}\cdot K +_2 \frac{1}{2}\cdot L+_2(-1,1)\right)&
\leq\Gsub{1}\Bigl(\bigl(-1,\sqrt{3.5}\,\bigr)\Bigr)=2<\sqrt{6.5}\\
&=\Media{2}{2}{3}{1/2}
=\Media{2}{\vphantom{\bigl(}\Gsub{1}(K)}{\Gsub{1}(L)}{1/2}.
\end{split}
\end{equation*}
\end{remark}

We observe now that Theorem \ref{t:p-B-M_discrete} holds also true for arbitrary non-negative
($L_p$) linear combinations of $K$ and $L$, but with the suitable
modification of the cube. More precisely, we have:
\begin{corollary}\label{c:disc_LpBM_t_s}
Let $t,s\geq0$ and $p\geq 1$, and let $K,L\subset\R^n$ be bounded sets such that
$\G(K)\G(L)>0$. Then
\[
\G\biggl(t\cdot K+_p s\cdot
L+\Bigl(-1,\bigl\lceil(t+s)^{1/p}\bigr\rceil\Bigr)^n\biggr)^{p/n}\geq
t\G(K)^{p/n}+s\G(L)^{p/n}.
\]
\end{corollary}
\begin{proof}
The proof follows the same argument to that of Theorem~\ref{t:Lp-BBL_discreta}, by replacing $(1-\lambda)$ and $\lambda$ by $t$ and $s$, respectively, for the characteristic functions $f=\chi_{_K}$, $g=\chi_{_L}$ and
$h=\chi_{_{t\cdot K +_p s\cdot L}}$.
So, in this case, it is enough to set
\[
\bar{t}=\bar{t}(\mu_0):=t^{1/p}(1-\mu_0)^{1/q}\quad\text{ and }\quad
\bar{s}=\bar{s}(\mu_0):=s^{1/p}\mu_0^{1/q},
\]
for which $\bar{t}+\bar{s}\leq(t+s)^{1/p}$ by H\"older's inequality, and then \[\Bigl(-1,\bigl\lceil(t+s)^{1/p}\bigr\rceil\Bigr)^n
\supset\Bigl(-1,\bigl\lceil \bar{t}+\bar{s}\bigr\rceil\Bigr)^n.\]
The proof is now concluded as in the one of Theorem~\ref{t:Lp-BBL_discreta}.
\end{proof}

In \cite{IYNZ} it was shown that if $A,B\subset\Z^n$ are finite, $A,B\neq\emptyset$, then
\begin{equation}\label{e:B-M_discrete_sum}
\bigl|A+B+\{0,1\}^n\bigr|^{1/n}\geq |A|^{1/n}+|B|^{1/n}.
\end{equation}
Here it makes no sense to wonder about an $L_p$ version of the above inequality, by just replacing $A+B$ by $A+_pB$ on the left-hand side, since $A+_pB$ is no longer finite (see \eqref{e:p-sum_extended}), for $p>1$. However, from  Corollary~\ref{c:disc_LpBM_t_s} for $K=A$, $L=B$ and $t=s=1$ we get the following result:

\begin{corollary}\label{t:Lp_B-M_discrete_cardinality}
Let $A,B\subset\Z^n$ be finite, $A,B\neq\emptyset$. Then
\begin{equation}\label{e:Lp_B-M_discrete_cardinality}
\G\bigl(A+_p B+(-1,2)^n\bigr)^{p/n}
\geq |A|^{p/n}+|B|^{p/n}.
\end{equation}
\end{corollary}

Clearly, for $p=1$, the latter inequality is exactly \eqref{e:B-M_discrete_sum}, since $A+B\subset\Z^n$ and the sole integer points in $(-1,2)^n$ are those in $\{0,1\}^n$.

\smallskip

We would like to note that unlike in the linear case ($p=1$), the cube
$(-1,2)^n$ on the left-hand side of \eqref{e:Lp_B-M_discrete_cardinality} cannot be,
in general, reduced to $\{0,1\}^n$ or even to $[0,1]^n$.

To see this, it is enough to consider $n=1$, $A=\{0,\dots,a\}$ and
$B=\{0,\dots,b\}$ for some $a,b\in\N$ with $0<a\leq b$. Indeed, on the one hand,
taking into account that $\Suma{p}{\cdot}{\cdot}{}$ is decreasing in $p$, we have
(see e.g. \cite[Theorem~19]{HaLiPo})
\[\Suma{p}{\vphantom{\bigl(}|A|}{|B|}{}=\Suma{p}{a+1}{b+1}{}\in[b+1,a+b+2],\]
and further $\Suma{p}{a+1}{b+1}{}>a+b+1$ for $p>1$ small enough.
On the other hand, if we denote by $K=[0,a]$ and $L=[0,b]$, then
$K+_pL=\bigl[0,\Suma{p}{a}{b}{}\bigr]$ since $K$ and $L$ are $1$-dimensional convex
bodies containing the origin (and thus their $p$-sum is also given by \eqref{e:def_p-sum}).
Moreover, due to the fact that $\Suma{p}{a}{b}{}<a+b$ for
all $p>1$, we obtain $\bigl\lfloor\Suma{p}{a}{b}{}\bigr\rfloor+1\leq a+b$.
Therefore, altogether we get
\begin{equation*}
\begin{split}
\Gsub1\bigl(A+_pB+[0,1]\bigr)&\leq\Gsub1\bigl(K+_pL+[0,1]\bigr)=\bigl\lfloor\Suma{p}{a}{b}{}\bigr\rfloor +2\\
&\leq a+b+1<\Suma{p}{a+1}{b+1}{}=\Suma{p}{\vphantom{\bigl(}|A|}{|B|}{}
\end{split}
\end{equation*}
for any $p>1$ small enough.
In fact, taking for instance $a=b=1$ and $p=3/2$, the latter inequality holds,
which shows that $[0,1]^n$ cannot replace $(-1,2)^n$ on the left-hand side of \eqref{e:Lp_B-M_discrete_cardinality}.

\smallskip
\subsection{From the discrete to the continuous case}
We will now prove Theorem~\ref{t:bblp_disc_to_cont}, i.e., we show that the discrete inequa\-lity collected in Theorem~\ref{t:Lp-BBL_discreta} implies the continuous result established in Theorem~\ref{t:L_p-BBL}, in the spirit of what happens for $p=1$ (see \cite[Theorem~2.4]{IYNZ}).

\begin{proof}[Proof of Theorem~\ref{t:bblp_disc_to_cont}]
Let $f,g,h:\R^n\longrightarrow\R_{\geq0}$
be functions in the conditions of Theorem \ref{t:L_p-BBL}, namely,
verifying \eqref{e:L_p-BBL_hyp}
for all $x,y\in\R^n$ with $f(x)g(y)>0$ and all $\mu\in[0,1]$, for some fixed $p\geq1$, $\lambda\in(0,1)$ and $-1/n\leq\alpha\leq\infty$.

\smallskip

We will first prove that, given $k\in\N$ and $C=[-k,k]^n$, we have
\begin{equation}\label{e:bblpdtoc_goal}
  \int_{C}h(z)dz\geq\Media{\frac{p\alpha}{n\alpha+1}}{\int_{C}f(x)\,\dlat x}{\int_{C}g(x)\,\dlat x}{\lambda}.
\end{equation}
Theorem~\ref{t:L_p-BBL} will then follow simply by taking limits as $k\to\infty$. To this aim, we may assume that the functions $f$, $g$ and $h$ vanish outside $C$ (multiplying them by the characteristic functions of $C$, if necessary). We shall also write $C_0=[-k,k)^n$.

For each $m\in\N$,
let $\Om=(-2^{-m},2^{-m})$ and $\Rm=[0,2^{-m})$, and
define the functions $f_m, g_m, h_m:\R^n\longrightarrow\R_{\geq0}$ given by
\[f_m(x)=\sup_{z\in x+\Rm}f(z), \,\text{ }\, g_m(x)=\sup_{z\in x+\Rm}g(z) \,\text{ and }\, h_m(x)=\sup_{z\in x+\Rm}h(z).\]
Moreover, for the sake of simplicity, we set $t:=(1-\lambda)^{1/p}(1-\mu)^{1/q}$ and $s:=\lambda^{1/p}\mu^{1/q}$ for any given $\mu\in[0,1]$, for which we get, as a consequence of H\"older's inequality, that $t+s\leq1$.
Again, condition \eqref{e:L_p-BBL_hyp} can be rewritten in terms of $t$, $s$ as
\[h(tz_1+sz_2)\geq\bigl[tf(z_1)^\alpha+sg(z_2)^\alpha\bigr]^{1/\alpha}\]
for all $z_1,z_2\in\R^n$ with $f(z_1)g(z_2)>0$.
Thus, since $(t+s)\Rm\subset\Rm$, we have
\begin{equation*}
\begin{split}
  h_m(tx+sy)&=\sup_{z\in tx+sy+\Rm}h(z)\geq\sup_{z\in t(x+\Rm)+s(y+\Rm)}h(z)\\
  &=\sup_{z_1\in x+\Rm,z_2\in y+\Rm}h(tz_1+sz_2)\\
  &\geq\sup_{z_1\in x+\Rm,z_2\in y+\Rm}\bigl[tf(z_1)^\alpha+sg(z_2)^\alpha\bigr]^{1/\alpha}\\
  &=\left[t\left(\sup_{z_1\in x+\Rm}f(z_1)\right)^\alpha+
  s\left(\sup_{z_2\in y+\Rm}g(z_2)\right)^\alpha\right]^{1/\alpha}\\
  &=\bigl[tf_m(x)^\alpha+sg_m(y)^\alpha\bigr]^{1/\alpha}
\end{split}
\end{equation*}
for all $x,y\in C$ (and so, in particular, for all $x,y\in C_0$) with $f_m(x)g_m(y)>0$ and all $\mu\in[0,1]$. Hence,
the functions $f_m,g_m,h_m$ are in the conditions of Corollary~\ref{c:Lp-BBL_disc_lattice} and we may apply it for the sets $K=L=C_0$ and the lattice $\2Z$. Note that in this case $\varphi\bigl((-1,1)^n\bigr)=\Om$ and thus we obtain
\begin{equation}\label{eq:bblpdtoc_thk_1}
\sum_{z\in[M_p+\Om]\cap\2Z}h_m^{\symbol_m}(z)\geq
\Media{\frac{p\alpha}{n\alpha+1}}{\sum_{x\in C_0\cap\2Z}f_m(x)}{\!\sum_{y\in C_0\cap\2Z}g_m(y)}{\lambda}\!\!,
\end{equation}
where $M_p=(1-\lambda)\cdot C_0+_p\lambda\cdot C_0$ and $h_m^{\symbol_m}(z)=\sup_{u\in\Om}h_m(z+u)$.
Now, since $C$ is an $n$-dimensional convex body containing the origin, from \eqref{e:def_p-sum} we get
\begin{equation*}
  (1-\lambda)\cdot C_0+_p\lambda\cdot C_0 \subset (1-\lambda)\cdot C +_p \lambda\cdot C=C,
\end{equation*}
which, jointly with the fact that $(C+\Om)\cap\2Z=C\cap\2Z$, allows us to deduce (from \eqref{eq:bblpdtoc_thk_1})
that
\begin{equation}\label{eq:bblpdtoc_thk_2}
\sum_{z\in C\cap\2Z}h_m^{\symbol_m}(z)\geq\Media{\frac{p\alpha}{n\alpha+1}}{\sum_{x\in C_0\cap\2Z}f_m(x)}{\sum_{y\in C_0\cap\2Z}g_m(y)}{\lambda}.
\end{equation}
We now consider the function $\overline{h}:\R^n\longrightarrow\R_{\geq0}$ given by $\overline{h}(x)=\sup_{\theta\in3\Om}h(x+\theta)$, and show that, for every fixed $z\in\R^n$ and any $x\in z+\Om$, we have $\overline{h}(x)\geq h_m^{\symbol_m}(z)$. Indeed,
\begin{equation}\label{eq:bblpdtoc_ineqh}
\begin{split}
    \overline{h}(x)&=\sup_{\theta\in3\Om}h(x+\theta)=\sup_{w\in\Om}\sup_{v\in\Om}\sup_{u\in\Om}h(x+u+v+w)\\
    &\geq\sup_{w\in\Om}\sup_{v\in\Om}\sup_{u\in\Rm}h(x+u+v+w)=\sup_{w\in\Om}\sup_{v\in\Om}h_m(x+v+w)\\
    &=\sup_{w\in\Om}h_m^{\symbol_m}(x+w)\geq h_m^{\symbol_m}(z).
\end{split}
\end{equation}
Furthermore, for any $r>0$ let
\begin{equation*}
\begin{split}
C_r&=\bigl\{x\in C:h(x)\geq r\bigr\}\text{ and}\\
\overline{C}_r&=\bigl\{x\in C+\Rm:\overline{h}(x)> r\bigr\}.
\end{split}
\end{equation*}
Notice that the superlevel sets $C_r$ are compact, since $h$ is upper semicontinuous and $C$ is compact (see \cite[Theorem~1.6]{RW}),
and then we clearly have $C_r=\bigcap_{m=1}^\infty(C_r+3\Om)$.
Moreover, since $h$ vanishes outside $C$, from the definition of $\overline{h}$ we get
$\overline{C}_r\subset C_r+3\Om$ for all $r>0$.
Thus, by Fubini's theorem and the monotone convergence theorem, we obtain
\begin{equation*}
\begin{split}
\int_{C}h(x)\,\dlat x &=\int_0^\infty\vol(C_r)\,\dlat r=\int_0^\infty\vol\left(\bigcap_{m=1}^\infty(C_r+3\Om)\right)\,\dlat r\\
&=\int_0^\infty\lim_{m\to\infty}\vol(C_r+3\Om)\,\dlat r\\
&=\lim_{m\to\infty}\int_0^\infty\vol(C_r+3\Om)\,\dlat r
\geq\lim_{m\to\infty}\int_0^\infty\vol(\overline{C}_r)\,\dlat r\\
&=\lim_{m\to\infty}\int_{C+\Rm}\overline{h}(x)\,\dlat x.
\end{split}
\end{equation*}
This, together with \eqref{eq:bblpdtoc_ineqh} and the fact that $C+\Rm=C\cap\2Z+\Rm$, implies that
\begin{equation*}
\int_{C}h(x)\,\dlat x \geq\lim_{m\to\infty}\int_{C+\Rm}\overline{h}(x)\,\dlat x \geq
\lim_{m\to\infty}2^{-mn}\sum_{z\in C\cap\2Z}h_m^{\symbol_m}(z).
\end{equation*}
Finally, since $f$ is Riemann integrable and $2^{-mn}\sum_{x \in C_0\cap\2Z} f_m(x)$ is an \emph{upper Riemann sum} of $f$ for the partition $\{x + \Rm: x\in C_0\cap\2Z\}$ of $C$,
we clearly have
\begin{equation*}
\lim_{m\rightarrow\infty} 2^{-mn}\sum_{x \in C_0\cap\2Z} f_m(x) = \int_C f(x)\dlat x.
\end{equation*}
The same holds for the function $g$ and then, taking limits on both sides of \eqref{eq:bblpdtoc_thk_2},
we get \eqref{e:bblpdtoc_goal}. This finishes the proof.
\end{proof}

Due to the well-known fact that a function is Riemann integrable if and only if it is continuous almost everywhere, and since the boundary of a convex set has null measure (and taking also into account the characterization of the upper semicontinuity in terms of the level sets), we directly get Theorem~\ref{t:L_pBM_disc_to cont}, as a consequence of Theorem~\ref{t:bblp_disc_to_cont}. We emphasize the necessity of assuming convexity in Theorem~\ref{t:L_pBM_disc_to cont}: if one considers bounded measurable sets $K, L \subset\R^n$ of positive volume, containing no rational point, one cannot expect to recover the $L_p$ Brunn-Minkowski inequality \eqref{ineq:lpbm} by shrinking the lattice $\Z^n$ by means of successively considering $2^{-m}\Z^n$, $m\in\N$.


\smallskip

\begin{thebibliography}{99}

\bibitem{Brt} F. Barthe, Autour de l'in\'egalit\'e de Brunn-Minkowski,
{\it Ann. Fac. Sci. Toulouse Math. (6)} {\bf 12} (2) (2003), 127--178.

\bibitem{Borell} C. Borell, Convex set functions in $d$-space, {\it Period.
Math. Hungar.} {\bf 6} (1975), 111--136.

\bibitem{BL} H. J. Brascamp and E. H. Lieb, On extensions of the
Brunn-Minkowski and Pr\'ekopa-Leindler theorems, including inequalities
for log concave functions and with an application to the diffusion
equation, {\it  J. Func. Anal.} {\bf 22} (4) (1976), 366--389.

\bibitem{Bu} P. S. Bullen, {\it Handbook of means and their
inequalities}. Mathematics and its Applications, 560, Revised from the
1988 original. Kluwer Academic Publishers Group, Dordrecht, 2003.


\bibitem{F62} Wm. J. Firey, $p$-means of convex bodies, {\it Math. Scand.} {\bf
10} (1962), 17--24.

\bibitem{G} R. J. Gardner, The Brunn-Minkowski inequality, {\it Bull. Amer. Math. Soc.}
{\bf 39} (3) (2002), 355--405.

\bibitem{GHW} R. J. Gardner, D. Hug and W. Weil, Operations between sets in geometry,
{\it J. Eur. Math. Soc.} {\bf 15} (2013), 2297--2352.

\bibitem{GHW2} R. J. Gardner, D. Hug and W. Weil, The Orlicz-Brunn-Minkowski theory: a general
framework, additions, and inequalities,
{\it J. Differential Geom.} {\bf 97} (2014), 427--476.

\bibitem{GG} R. J. Gardner and P. Gronchi, A Brunn-Minkowski inequality for the integer lattice,
{\it Trans. Amer. Math. Soc.} {\bf 353} (10) (2001), 3995--4024.

\bibitem{GT} B. Green and T. Tao, Compressions, convex geometry and the Freiman-Bilu theorem,
{\it Q. J. Math.} {\bf 57} (4) (2006), 495--504.

\bibitem{HKS} D. Halikias, B. Klartag and B. A. Slomka, Discrete variants of Brunn-Minkowski type inequalities,
{\it Submitted}, \href{https://arxiv.org/abs/1911.04392}{arXiv:1911.04392}.

\bibitem{HaLiPo} G. H. Hardy, J. E. Littlewood and G. P{\'o}lya, {\it Inequalities}.
Cambridge Mathematical Library, Reprint of the 1952 edition. Cambridge University Press, Cambridge, 1988.

\bibitem{HCIYN} M. A. Hern\'andez Cifre, D. Iglesias and J. Yepes Nicol\'as, On a discrete Brunn-Minkowski type
inequality, {\it SIAM J. Discrete Math.} {\bf 32} (2018), 1840--1856.

\bibitem{IYN} D. Iglesias and J. Yepes Nicol\'as, On discrete Borell-Brascamp-Lieb inequalities,
{\it Rev. Matem\'atica Iberoamericana} {\bf 36} (3) (2020), 711--722.

\bibitem{ILYN} D. Iglesias, E. Lucas and J. Yepes Nicol\'as, On discrete
Brunn-Minkowski and isoperimetric type inequalities, {\it Submitted}.

\bibitem{IYNZ} D. Iglesias, J. Yepes Nicol\'as and A. Zvavitch, Brunn-Minkowski type
inequalities for the lattice point enumerator, to appear in \textit{Adv. Math.}

\bibitem{KL} B. Klartag and J. Lehec, Poisson processes and a log-concave Bernstein theorem,
{\it Stud. Math.} {\bf 247} (1) (2019), 85--107.

\bibitem{L93} E. Lutwak, The Brunn-Minkowski-Firey theory, I: Mixed volumes and the
Minkowski problem, {\it J. Differential Geom.} {\bf 38} (1993), 131-150.

\bibitem{L96} E. Lutwak, The Brunn-Minkowski-Firey theory, II: Affine and geominimal surface
areas, {\it Adv. Math.} {\bf 118} (1996), 244--294.

\bibitem{LYZ} E. Lutwak, D. Yang and G. Zhang, The
Brunn-Minkowski-Firey inequality for nonconvex sets, {\it Adv.
Appl. Math.} {\bf 48} (2012), 407--413.

\bibitem{Me} T. Mesikepp, $M$-Addition, {\it J. Math. Anal. Appl.} {\bf 443} (2016), 146--177.

\bibitem{RW} R. T. Rockafellar and R. J.-B. Wets, {\it Variational
analysis}. Grundlehren der Mathematischen Wissenschaften [Fundamental
Principles of Mathematical Sciences], 317. Springer-Verlag, Berlin, 1998.


\bibitem{RoXi} M. Roysdon and S. Xing, On $L_p$-Brunn-Minkowski type and $ L_p$-isoperimetric type inequalities for
measures, to appear in {\it Trans. Amer. Math. Soc.} 

\bibitem{Ru1} I. Z. Ruzsa, Sum of sets in several dimensions,
{\it Combinatorica} {\bf 14} (1994), 485--490.

\bibitem{Ru2} I. Z. Ruzsa, Sets of sums and commutative graphs,
{\it Studia Sci. Math. Hungar.} {\bf 30} (1995), 127--148.

\bibitem{Sch2} R. Schneider, {\it Convex bodies: The Brunn-Minkowski
theory}. 2nd expanded ed. Encyclopedia of Mathematics and its
Applications, 151. Cambridge University Press, Cambridge, 2014.

\bibitem{Sl} B. A. Slomka, A Remark on discrete Brunn-Minkowski type inequalities via
transportation of measure, {\it Submitted}, \href{https://arxiv.org/abs/2008.00738}{arXiv:2008.00738}.

\bibitem{Wu} Y. Wu, A Pr\'ekopa-Leindler type inequality related to the $L_p$ Brunn-Minkowski
inequality, {\it Submitted}, \href{https://arxiv.org/abs/2007.01101}{arXiv:2007.01101}.

\end{thebibliography}
\end{document}